\documentclass{article}

\usepackage[utf8]{inputenc} % allow utf-8 input
\usepackage[T1]{fontenc}    % use 8-bit T1 fonts
\usepackage{hyperref}       % hyperlinks
\usepackage{url}            % simple URL typesetting
\usepackage{booktabs}       % professional-quality tables
\usepackage{amsfonts}       % blackboard math symbols
\usepackage{nicefrac}       % compact symbols for 1/2, etc.
\usepackage{microtype}      % microtypography
\usepackage{xcolor}         % colors
\usepackage{mathdots}
\usepackage{amsthm,amsmath,amssymb}
\usepackage{geometry}
\usepackage{natbib}
\usepackage{fullpage}

\def\showauthornotes{1}
\def\showdraftbox{1}

\renewcommand{\epsilon}{\varepsilon}
\newcommand{\eps}{\varepsilon}

\usepackage{hyperref}
\usepackage{bbm,xspace,nicefrac,amsthm}
\usepackage{color,enumerate,bm,verbatim,textcomp,float}
\usepackage[small]{caption}
\usepackage{comment} 
\usepackage{algorithm} %ctan.org\pkg\algorithms
\usepackage[noend]{algpseudocode}
\usepackage{thmtools}
\usepackage{thm-restate}
\usepackage{fancybox}
\usepackage[shortlabels]{enumitem}
\usepackage{booktabs}
\usepackage{commath}
\newcommand{\defeq}{\stackrel{\textup{def}}{=}}

\newtheorem{theorem}{Theorem}[section]
\newtheorem{lemma}[theorem]{Lemma}
\newtheorem{assumption}[theorem]{Assumption}
\newtheorem{definition}[theorem]{Definition}
\newcommand{\diag}[1]{{\bf Diag}\left({#1}\right)}

\newcommand{\nfrac}[2]{\nicefrac{#1}{#2}}
\def\abs#1{\left| #1 \right|}
\renewcommand{\norm}[1]{\ensuremath{\left\lVert #1 \right\rVert}}

\newcommand\rea{\mathbb R}

\newcommand{\marginlabel}[1]%
{\mbox{}\marginpar{\it{\raggedleft\hspace{0pt}#1}}}
\newcommand\calM{\mathcal{M}}

\newcommand\calR{\mathcal{R}}

\newcommand\calX{\mathcal{X}}
\newcommand\calY{\mathcal{Y}}
\newcommand\calZ{\mathcal{Z}}

\definecolor{Mygray}{gray}{0.8}

 \ifcsname ifcommentflag\endcsname\else
  \expandafter\let\csname ifcommentflag\expandafter\endcsname
                  \csname iffalse\endcsname
\fi

\ifnum\showauthornotes=1
\newcommand{\Authornote}[2]{{\sf\small\color{red}{[#1: #2]}}}
\newcommand{\Authoredit}[2]{{\sf\small\color{red}{[#1]}\color{blue}{#2}}}
\newcommand{\Authorcomment}[2]{{\sf \small\color{gray}{[#1: #2]}}}
\newcommand{\Authorfnote}[2]{\footnote{\color{red}{#1: #2}}}
\newcommand{\Authorfixme}[1]{\Authornote{#1}{\textbf{??}}}
\newcommand{\Authormarginmark}[1]{\marginpar{\textcolor{red}{\fbox{%\Large
#1:!}}}}
\else
\newcommand{\Authornote}[2]{}
\newcommand{\Authoredit}[2]{}
\newcommand{\Authorcomment}[2]{}
\newcommand{\Authorfnote}[2]{}
\newcommand{\Authorfixme}[1]{}
\newcommand{\Authormarginmark}[1]{}
\fi

\newlength{\pgmtab}  %  \pgmtab is the width of each tab in the
\let\originalleft\left
\let\originalright\right
\renewcommand{\left}{\mathopen{}\mathclose\bgroup\originalleft}
  \renewcommand{\right}{\aftergroup\egroup\originalright}

\def\defeq{\stackrel{\mathrm{def}}{=}}

\def\diag#1{\textsc{Diag}\left( #1 \right)}

\def\aa{\pmb{\mathit{a}}}
\newcommand\bb{\boldsymbol{\mathit{b}}}

\newcommand\dd{\boldsymbol{\mathit{d}}}
\newcommand\ee{\boldsymbol{\mathit{e}}}
\newcommand\ff{\boldsymbol{\mathit{f}}}

\renewcommand\gg{\mathit{F}}
\newcommand\hh{\boldsymbol{\mathit{h}}}

\newcommand\uu{\boldsymbol{\mathit{u}}}
\newcommand\vv{\boldsymbol{\mathit{v}}}

\newcommand\yy{\boldsymbol{\mathit{y}}}
\newcommand\zz{\boldsymbol{\mathit{z}}}
\newcommand\xx{\boldsymbol{\mathit{x}}}

\renewcommand\AA{\boldsymbol{\mathit{A}}}
\newcommand\BB{\boldsymbol{\mathit{B}}}
\newcommand\CC{\boldsymbol{\mathit{C}}}

\newcommand\GG{\boldsymbol{\mathit{G}}}
\newcommand\II{\boldsymbol{\mathit{I}}}

\newcommand\Otil{\widetilde{O}}

 {
	\begin{enumerate}}{\end{enumerate}}

\ifnum\showdraftbox=1

\else

\fi

\newcommand\restrict[1]{\raisebox{-.5ex}{$\big|$}_{#1}}

\newcommand\bI{\mathbf{I}}

\title{Optimal Methods for Higher-Order Smooth Monotone\\Variational Inequalities}

\author{%
 Deeksha Adil\\
  Department of Computer Science\\
  University of Toronto\\
  \texttt{deeksha@cs.toronto.edu}
   \and
  Brian Bullins \\
  Toyota Technological Institute at Chicago \\
  \texttt{bbullins@ttic.edu}\\ 
    \and
  Arun Jambulapati \\
  ICME \\
  Stanford University \\
  \texttt{jmblpati@stanford.edu} 
  \and
  Sushant Sachdeva \\
  Department of Computer Science \\
  University of Toronto \\
  \texttt{sachdeva@cs.toronto.edu}
}

\begin{document}

\maketitle

\begin{abstract}
In this work, we present new simple and optimal algorithms for solving
the variational inequality (VI) problem for $p^{th}$-order smooth,
monotone operators --- a problem that generalizes convex optimization
and saddle-point problems.
Recent works (Bullins and Lai (2020), Lin and Jordan (2021), Jiang and
Mokhtari (2022)) present methods that achieve a rate of
$\Otil(\epsilon^{-2/(p+1)})$ for $p\geq 1$, extending results by
(Nemirovski (2004)) and (Monteiro and Svaiter (2012)) for $p=1,2$.
A drawback to these approaches, however, is their reliance on a line
search scheme.
We provide the first $p^{th}$-order method that achieves a
rate of $O(\epsilon^{-2/(p+1)}).$ Our method does not rely on a line
search routine, thereby improving upon previous rates by a logarithmic
factor. Building on the Mirror Prox method of Nemirovski (2004), our
algorithm works even in the constrained, non-Euclidean setting.
Furthermore, we prove the optimality of our algorithm by constructing
matching lower bounds. These are the first lower bounds for smooth
MVIs beyond convex optimization for $p > 1$.
This establishes a separation between solving smooth MVIs and smooth
convex optimization, and settles the oracle complexity of solving
$p^{\textrm{th}}$-order smooth MVIs.
\end{abstract}

\section{Introduction}

In the variational inequality (VI) problem, given an operator $\gg:\calZ \rightarrow \mathbb{R}^n$ over a closed convex set $\calZ \subseteq \mathbb{R}^n$, the goal is to find $\zz^{\star}\in \calZ$ that satisfies:
\[
\langle \gg(\zz),\zz^{\star}-\zz\rangle \leq 0, \quad \forall \zz \in \calZ.
\]
This problem captures constrained convex optimization by setting $\gg$ to be the gradient of the function, as well as min-max problems of the form
\[
\min_{\xx\in \calX}\max_{\yy \in \calY} \quad \phi(\xx,\yy)
\]
 by setting $\gg = \begin{bmatrix}\nabla_{\xx}\phi, -\nabla_{\yy}\phi\end{bmatrix}^{\top}$ for $\zz = (\xx,\yy).$ 
The VI problem has proven itself useful across a wide range of applications which include training neural networks \citep{madry2018towards} and generative adversarial networks (GANs) \citep{goodfellow2014generative}, signal processing \citep{liu2013max,giannakis2016decentralized}, as well as game theoretic applications such as for finding Nash equilibria \citep{daskalakis2011near}.

In this work, we focus on simple and optimal algorithms for the case
of {\it monotone} operators and the associated {\it monotone
  variational inequality} (MVI) problem which generalizes convex
optimization to the VI setting. MVIs capture convex-concave saddle
point problems, and include applications from robust optimization~\citep{ben2009robust} and
zero-sum games~\citep{kroer2018solving}.

In the special case of convex optimization, restricting to smooth
convex functions (with bounded Lipschitz constant of the function
gradient) allows us to obtain fast convergent algorithms with an
iteration complexity of $O(\epsilon^{-1/2})$,
e.g. Nesterov's accelerated gradient
descent~\citep{nesterov1983method,nesterov2004introductory}, which is optimal in this setting.
Analogously, for smooth $(p=1)$ MVIs, the Mirror Prox method of
\citet{nemirovski2004prox} and the dual exterapolation method of
\citet{Nes2007dual} achieve an $O(\epsilon^{-1})$ iteration complexity,
building on the initial extragradient method of
\citet{korpelevich1976extragradient}.
This rate has been shown to be tight for MVIs, assuming access to only
a first order oracle, for smooth convex-concave saddle point problems
\citep{ouyang2021lower}, which, as we have seen, are a special case of
the MVI problem.

In the search for better algorithms for convex optimization, recent
celebrated works have obtained methods with improved convergence rates
of $\Otil(\epsilon^{-{2}/{(3p+1)}}),$ where
$\Otil(\cdot)$ hides logarithmic factors,~\citep{
  monteiro2013accelerated, gasnikov2019near,song2021unified}. These methods assume
smoothness of $p^{th}$-order derivatives and access to an oracle that
minimizes a regularized $p^{th}$-order Taylor series expansion of the
function.
These methods have again been shown to be optimal for convex
optimization by giving matching lower bounds (up to logarithmic
factors) assuming access to only a $p^{th}$-order Taylor series
oracle~\citep{agarwal2018lower, arjevani2019oracle}.

It is natural to ask if higher-order smoothness assumptions can allow
for algorithms for solving MVIs with improved convergence rates.
Inspired by the cubic regularization method
\citep{nesterovpolyak2006cubic}, \citet{Nes06Cubic} considers a
second-order approach for MVIs where the Jacobian of the operator is
Lipschitz continuous ($p=2$), and achieves an $O(\epsilon^{-1})$
rate. Under the same second-order smoothness assumption,
\citet{monteiro2012iteration} show how to achieve an improved
convergence rate of $O(\epsilon^{-2/3})$.
For $p^{th}$-order smooth MVIs, recent works
\citep{bullins2020higher,lin2021capturing, jiang2022generalized} have
established convergence rates of $\Otil(\epsilon^{-2/(p+1)}),$ again
assuming access to a $p^{th}$-order oracle. Note that this rate is
strictly worse than that for convex optimization.

A drawback of all these algorithms for higher-order smooth MVIs,
including \cite{monteiro2012iteration}, is that they require a line
search procedure. The first question we address is whether such a
line-search is necessary, or if one can design a simpler line-search-free algorithm for $p^{th}$-order smooth MVIs without compromising on
the iteration count.

More importantly, there are no matching lower bounds for solving
$p^{th}$-order smooth MVIs. Thus, it is unknown whether a convergence rate
of $\Otil(\epsilon^{-2/(p+1)})$ is optimal for $p^{th}$-order smooth
MVIs, or if one could hope to achieve better rates, possibly matching
those for convex optimization.

\paragraph{Our Results.}
In this work, we provide a simple algorithm for solving $p^{th}$-order
smooth MVIs which achieves a rate of $O(\epsilon^{-2/(p+1)})$ without
requiring any line-search procedure, thereby improving upon previous
works by a logarithmic factor. Our algorithm builds on the
Mirror Prox approach of \citet{nemirovski2004prox}, resulting in a much more simplified analysis compared to the previous line-search-dependent methods.
In addition, our algorithm is applicable to both non-Euclidean and
constrained settings.
Our algorithm requires access to an oracle for solving an MVI
subproblem (see Definition~\ref{def:VISub}) obtained by regularizing
the $p^{th}$-order Taylor series expansion for the operator.
This is analogous to the Taylor series oracle from the works on
highly-smooth convex optimization~\citep{bubeck2019near,
  gasnikov2019near}, and identical to the oracle from the
\cite{jiang2022generalized} work on highly-smooth VIs.

Additionally, we construct a family of hard saddle-point problems, and
we show that every algorithm that has access to only a $p^{th}$-order
Taylor series oracle will require $\Omega(\epsilon^{-2/(p+1)})$
iterations to converge.
To the best of our knowledge, this is the first lower bound for
$p^{th}$-order smooth MVI problems for $p\geq 2,$ and it shows that our
algorithm is optimal up to constant factors. This effectively settles the oracle
complexity of highly-smooth MVIs, and it furthermore establishes a separation from the
minimization of highly-smooth convex functions.

\paragraph{Approximately Solving Subproblems.}
The $p^{th}$-order MVI subproblems (Definition~\ref{def:VISub}) that need to be solved in our
algorithm are identical to those arising
the in the algorithm from \citet{jiang2022generalized}, and when
restricted to the case of unconstrained convex optimization with
smoothness measured in Euclidean norms, they become identical to those
from \citet{bubeck2019near}.
In the appendix, we show that it is sufficient to solve the
subproblems approximately. Further we show how to solve the subproblem
efficiently in the $p=2$ case.

All previous works on higher-order algorithms \citep{gasnikov2019near,
  bubeck2019near, jiang2022generalized}
assume access to an oracle for solving such subproblems.
Even for the special case of unconstrained convex optimization and
Euclidean norms, it remains an open problem for how to solve these
subproblems for $p\geq 3$.

\paragraph{Independent Work \citep{lin2022perseus}} A concurrent work
by \cite{lin2022perseus} also presents an algorithm for
$p^{th}$-order smooth MVIs that does not require a binary
search procedure and achieves a rate of $O(\epsilon^{-2/(p+1)}).$
Their work builds on the dual extrapolation method and solves the same
subproblems as our algorithm. Their algorithm is also shown to work
only for Euclidean norms, although it can possibly be extended to
non-Euclidean settings as well. Our results were derived independently
and our algorithm works for non-Euclidean settings. Additionally, we
include lower bounds which establish that these rates are
optimal. \cite{lin2022perseus} also make note of the keen observation
by \cite[Section 4.3.3]{nesterov2018lectures} that eliminating the
binary search could provide significant practical benefit (relative to
the improvements in terms of $\epsilon$), and thus being able to do so
has remained a key open problem.

\section{Preliminaries}

We let $\calX ,\calY,\calZ \subseteq \rea^n$ denote closed convex sets. We use $F:\calX \rightarrow \rea^n$ to denote an operator, $\mathbb{R}^n_{k}$ to denote the space of $\xx \in \mathbb{R}^n$ with $\xx_i = 0, \forall i > k$, and $\ee_{i,n}$ to denote the all $0$'s vector with $1$ at the $i^{th}$ coordinate. We let $\|\cdot\|$ denote any norm and $\dd : \calX \rightarrow \rea$ denotes a prox function that is strongly convex with respect to $\|\cdot\|$, i.e.,
\[
\dd(\xx) - \dd(\yy) - \langle \nabla\dd(\yy),\xx-\yy\rangle \geq \|\xx-\yy\|^2.
\]
Let $\omega(\xx,\yy)$ denote the Bregman divergence of $\dd$,
i.e.,
\begin{equation}\label{eq:Bregman}
\omega(\xx,\yy)   =  \dd(\xx) - \dd(\yy) - \langle \nabla\dd(\yy),\xx-\yy\rangle \geq \|\xx-\yy\|^2.
\end{equation}

\subsection{Standard Results}

We first recall several standard results which will be useful throughout the paper, starting with the three point property of the Bregman divergence, which generalizes the law of cosines.

\begin{lemma}[Three Point Property]\label{lem:3point}
Let $\omega(\xx,\yy)$ denote the Bregman divergence of a function $\dd$. The \textit{three point property} states, for any $\xx,\yy,\zz$,
\[
\langle \nabla \dd(\yy) - \nabla \dd(\zz), \xx-\zz \rangle = \omega(\xx,\zz) + \omega(\zz,\yy) - \omega(\xx,\yy).
\]
\end{lemma}

\begin{lemma}[\cite{tseng2008accelerated}]\label{lem:Tseng}
Let $\phi$ be a convex function, let $\xx\in \calX$, and let
\[
\xx_+ = \arg\min_{\yy\in \calX}\{\phi(\yy) + \omega(\yy,\xx)\}.
\]
Then, for all $\yy\in \calX$, we have,
$ \phi(\yy) + \omega(\yy,\xx) \geq \phi(\xx_+) + \omega(\xx_+,\xx) +
\omega(\yy,\xx_+).$
\end{lemma}

The next lemma follows from the power mean inequality (see \citep[Lemma 4.4]{bullins2020higher}).
\begin{lemma}\label{lem:Bullins}
  Given $R, \xi_1,\ldots, \xi_T \ge 0$ such that 
$\sum_{t=1}^T \xi_t^2 \leq R,$ we have   $\sum_{t = 1}^T \xi_t^{-q} \geq \frac{T^{\frac{q}{2}+1}}{R^{\frac{q}{2}}}$.
\end{lemma}

\subsection{Monotone Variational Inequalities}
\label{sec:Monotone}

In this section, we formally define our problem and some definitions for higher-order derivatives.

\begin{definition}[Directional Derivative]
Let $\calX \subseteq \mathbb{R}^n$. Consider a $k$-times differentiable operator $F:\calX \rightarrow \mathbb{R}^n$. For $r \leq k+1$, we let
\begin{equation*}
\nabla^{k} F(\xx) [\hh]^{r} = \frac{\partial^{k}}{\partial h^{k}}\restrict{t_1=0,\dots,t_{r}=0} F(x+t_1 \hh + \dots + t_{r}\hh)
\end{equation*}
denote, for $\xx,\hh\in \calX$, the $k^{th}$ directional derivative of a $F$ at $\xx$ along $\hh$.
\end{definition}

\begin{definition}[Monotone Operator]
For $\calX \subseteq \mathbb{R}^n$, consider an operator $\gg:\calX \rightarrow \mathbb{R}^n$. We say that $\gg$ is monotone if
\[
\forall \xx,\yy \in \calX, \quad \langle \gg(\xx) - \gg(\yy) ,\xx-\yy\rangle \geq 0.
\]
Equivalently, an operator $\gg$ is monotone if its \textit{Jacobian} $\nabla\gg$ is positive semidefinite. 
\end{definition}

\begin{definition}[Higher-Order Smooth Operator]\label{def:Lipschitz}
For $p\geq 1$, an operator $\gg$ is $p^{th}$-order $L_p$-smooth with respect to a norm $\|\cdot \|$
if the higher-order derivative of $\gg$ satisfies
\[
\|\nabla^{p-1}\gg(\yy)-\nabla^{p-1}\gg(\xx)\|_{*} \leq L_p\|\yy-\xx\|, \forall \xx,\yy \in \calX,
\]
or
\[
  \|\gg(\yy)- \mathcal{T}_{p-1}(\yy;\xx)\|_{*} \leq \frac{L_p}{p!} \|\yy-\xx\|^p,
\]
 where we let
 \[
 \mathcal{T}_p(\yy;\xx) = \sum_{i =0}^{p} \frac{1}{i!}\nabla^i \gg(\xx)[\yy-\xx]^{i},
 \]
  denote the $p^{th}$-order Taylor expansion of $\gg$, and we let 
\begin{equation*}
\|\nabla^{p-1} \gg(\yy) - \nabla^{p-1} \gg(\xx)\|_* \defeq \max\limits_{\hh : \norm{\hh} \leq 1} |\nabla^{p-1} \gg(\yy) [\hh]^{p} - \nabla^{p-1} \gg(\xx) [\hh]^{p}|
\end{equation*}
 denote the operator norm.
\end{definition}
For any operator $\gg$, the variational inequality problem associated with $\gg$ may ask for two kinds of solutions which we define next.
\begin{definition}[Weak and Strong Solutions]
For $\calX \subseteq \mathbb{R}^n$ and operator $\gg:\calX \rightarrow \mathbb{R}^n$, a strong solution to the variational inequality problem associated with $\gg$ is a point $\xx^{\star} \in \calX$ satisfying:
\[
\langle \gg(\xx^{\star}) , \xx^{\star}-\xx \rangle \leq 0, \quad \forall \xx \in \calX.
\]
A weak solution to the variational inequality problem associated with $\gg$ is a point $\xx^{\star} \in \calX$ satisfying:
\[
\langle \gg(\xx) , \xx^{\star}-\xx \rangle \leq 0, \quad \forall \xx \in \calX.
\]
If $\gg$ is continuous and monotone, then a weak solution is the same as a strong solution.
\end{definition}

\begin{definition}[$\epsilon$-Approximate MVI Solution]\label{def:Problem} Let $\epsilon>0$, $\calX \subseteq \mathbb{R}^n$, and operator $\gg:\calX \rightarrow \mathbb{R}^n$ be monotone, continuous and $p^{th}$-order $L_p$-smooth with respect to a norm $\|\cdot \|$. 
Our goal is to find an $\epsilon$-approximate solution to the MVI, i.e., an $\xx^{\star} \in \calX$ satisfying:
\[
\langle \gg(\xx) , \xx^{\star}-\xx \rangle \leq \epsilon, \quad \forall \xx \in \calX.
\]
\end{definition}

\paragraph{Organization.} In Section \ref{sec:algo} we present our algorithm and analysis for the MVI problem (Definition \ref{def:Problem}). In Section \ref{sec:LowerBound} we present a lower bound for the MVI problem which shows that our rates of convergence are tight up to constant factors. We further show how to solve our MVI subproblem for $p=2$, the details of which we defer to the appendix.

\section{Algorithm}\label{sec:algo}

We now present our algorithm for the MVI problem defined in Definition \ref{def:Problem}. Our algorithm is based on a Mirror Prox method and does not require any binary search procedure or solution to an implicit subproblem.

Our algorithm \textsc{MVI-OPT} (Algorithm \ref{alg:MainAlgo}) solves the following subproblem at every iteration. 

\begin{definition}[MVI Subproblem]\label{def:VISub}
We assume access to an oracle which, for any $\hat{\xx} \in \mathcal{X}$, solves the following variational inequality problem:
\[
\text{Find }\ T(\hat{\xx}):\ \langle U_{p,\hat{\xx}}(T(\hat{\xx})), T(\hat{\xx})-\xx\rangle \leq 0, \quad \forall \xx \in \calX,
\]
where 
\[
U_{p,\xx}(\yy) = \mathcal{T}_{p-1}(\yy;\xx) + \frac{2L_p}{p!} \omega(\yy,\xx)^{\frac{p-1}{2}}\left(\nabla \dd(\yy) - \nabla \dd(\xx)\right).
\]
\end{definition}

We note that for the case of $\calX =\mathbb{R}^n$, $\dd(\xx) = \|\xx\|_2^2$ and $\gg = \nabla \ff$, where $\ff$ is a $p^{th}$-order smooth convex function, the above subproblem is equivalent to the subproblem solved by the algorithm of \cite{bubeck2019near} (up to constant factors), which is known to have optimal iteration complexity for highly-smooth convex optimization. Previous works on higher-order smooth MVIs also solve essentially the same subproblem in their algorithms \citep{jiang2022generalized}. It has been shown by \cite[Lemma 7.1]{jiang2022generalized} that these subproblems are monotone and are guaranteed to have a unique solution, though efficiently finding such a solution in general remains an open problem, even in the case of convex optimization. We further show in the appendix that it is sufficient to solve these subproblems approximately, and for the case of $p=2$ we provide an algorithm for solving the associated subproblem.

\begin{algorithm}
\caption{Algorithm for Higher-Order Smooth MVI Optimization}\label{alg:MainAlgo}
 \begin{algorithmic}[1]
\Procedure{MVI-OPT}{$\xx_0 \in \calX, K, p$}
\For{ $i=0$ to $i=K$}
  \State $\xx_{i+\frac{1}{2}}  \leftarrow T(\xx_i)$
  \State $\lambda_i \leftarrow \frac{1}{2}\omega(\xx_{i+\frac{1}{2}},\xx_i)^{-\frac{p-1}{2}}$
  \State $\xx_{i+1} \leftarrow \arg\min_{\xx \in \calX} \left\{\langle F(\xx_{i+\frac{1}{2}}),\xx -\xx_{i+\frac{1}{2}}\rangle  + \frac{L_p}{p!\lambda_i} \omega(\xx,\xx_i)\right\}$
\EndFor
\State\Return $\hat{\xx}_K = \frac{\sum_{i=0}^K \lambda_i \xx_{i+\frac{1}{2}}}{\sum_{i=0}^K \lambda_i}$
\EndProcedure 
 \end{algorithmic}
\end{algorithm}
We now move to the analysis of our algorithm. The following lemma, which helps us prove our final rate of convergence, characterizes the iterates and step sizes involved in our algorithm.
\begin{lemma}\label{lem:Induction}
For any $K \geq 1$ and $\xx \in \calX$, the iterates $\xx_{i + \frac 12}$ and parameters $\lambda_i$ satisfy
\[
\sum_{i=0}^K  \lambda_i \frac{p!}{L_p} \langle F(\xx_{i+\frac{1}{2}}),\xx_{i+\frac 12 } -\xx\rangle \leq \omega(\xx, \xx_0) - \frac{15}{16} \sum_{i=0}^K \left(2 \lambda_i\right)^{- \frac{2}{p-1}}.
\]
\end{lemma}
\begin{proof}
For any $i$ and any $\xx \in \calX$, we first apply Lemma~\ref{lem:Tseng} with $\phi(\xx) = \lambda_i \frac{p!}{L_p} \langle F(\xx_{i+\frac{1}{2}}),\xx - \xx_i\rangle$, which gives us
\begin{equation}\label{eq:Tsengsec3}
\lambda_i \frac{p!}{L_p} \langle F(\xx_{i+\frac{1}{2}}),\xx_{i+1} -\xx\rangle \leq \omega(\xx,\xx_i) - \omega(\xx,\xx_{i+1}) - \omega(\xx_{i+1},\xx_i). 
\end{equation}
Additionally, the guarantee of Definition~\ref{def:VISub} with $\xx = x_{i+1}$ yields
\begin{equation}
\label{eqn:MPstep1sec3}
\left\langle \mathcal{T}_{p-1}(\xx_{i+\frac 12}; \xx_i) , \xx_{i+\frac 12} - \xx_{i+1} \right\rangle  \leq \frac{2 L_p}{p!} \omega(\xx_{i+\frac 12}, \xx_i)^{\frac{p-1}{2}} \left\langle \nabla \dd(\xx_i) - \nabla \dd(\xx_{i+\frac 12}),  \xx_{i+\frac 12} - \xx_{i+1}  \right\rangle.
\end{equation}
Applying the Bregman three point property (Lemma~\ref{lem:3point}) and the definition of $\lambda_k$ to Equation~\ref{eqn:MPstep1sec3}, we have
\begin{equation}
\label{eqn:Tseng2sec3}
\lambda_i \frac{p!}{L_p}  \left\langle \mathcal{T}_{p-1}(\xx_{i+\frac 12}; \xx_i) , \xx_{i+\frac 12} - \xx_{i+1} \right\rangle  \leq \omega(\xx_{i+1}, \xx_i) - \omega(\xx_{i+1}, \xx_{i+\frac 12}) - \omega(\xx_{i+\frac 12}, \xx_i).
\end{equation}

Summing Equations~\ref{eq:Tsengsec3} and \ref{eqn:Tseng2sec3}, we obtain 
\begin{align}
  \label{eqn:one_iter_1sec3}
&\lambda_i \frac{p!}{L_p} \left( \langle F(\xx_{i+\frac{1}{2}}),\xx_{i+\frac 12 } -\xx\rangle + \left\langle   \mathcal{T}_{p-1}(\xx_{i+\frac 12}; \xx_i) - F(\xx_{i + \frac 12 }) , \xx_{i+\frac 12} - \xx_{i+1} \right\rangle \right) \nonumber \\
&\leq \omega(\xx,\xx_i) - \omega(\xx,\xx_{i+1}) - \omega(\xx_{i+1}, \xx_{i+\frac 12}) - \omega(\xx_{i+\frac 12}, \xx_i).
\end{align}

Now, we obtain 
\begin{align*}
\lambda_i &\frac{p!}{L_p} \left\langle   \mathcal{T}_{p-1}(\xx_{i+\frac 12}; \xx_i) - F(\xx_{i + \frac 12 }) , \xx_{i+\frac 12} - \xx_{i+1} \right\rangle \\
&\substack{(i) \\ \geq} - \lambda_i \frac{p!}{L_p} \norm{\mathcal{T}_{p-1}(\xx_{i+\frac 12}; \xx_i) - F(\xx_{i + \frac 12 }) }_* \norm{ \xx_{i+\frac 12} - \xx_{i+1}} \\
&\substack{(ii) \\ \geq} - \lambda_i \norm{\xx_{i+ \frac 12} -  \xx_i}^p \norm{ \xx_{i+\frac 12} - \xx_{i+1}} \\ 
&\substack{(iii) \\ \geq} -\frac{1}{2} \omega(\xx_{i + \frac 12}, \xx_i)^{-\frac{p-1}{2}} \omega(\xx_{i + \frac 12}, \xx_i)^{\frac{p}{2}} \omega(\xx_{i+1}, \xx_{i+\frac 12})^{\frac{1}{2}} \\ 
&= - \frac{1}{2} \omega(\xx_{i + \frac 12}, \xx_i)^{\frac{1}{2}} \omega(\xx_{i+1}, \xx_{i+\frac 12})^{\frac{1}{2}} \\
&\substack{(iv) \\ \geq} - \frac{1}{16}  \omega(\xx_{i + \frac 12}, \xx_i) - \omega(\xx_{i+1}, \xx_{i+\frac 12}).
\end{align*}
Here, $(i)$ used H\"older's inequality, $(ii)$ used  Definition~\ref{def:Lipschitz}, $(iii)$ used the $1$-strong convexity of $\omega$, and $(iv)$ used the inequality $\sqrt{xy} \leq 2x + \frac{1}{8} y$ for $x,y \geq 0$. Combining with \ref{eqn:one_iter_1sec3} and rearranging yields
\begin{equation*}
\lambda_i \frac{p!}{L_p} \langle F(\xx_{i+\frac{1}{2}}),\xx_{i+\frac 12 } -\xx\rangle \leq \omega(\xx,\xx_i) - \omega(\xx,\xx_{i+1}) - \frac{15}{16} \omega(\xx_{i+\frac 12}, \xx_i). 
\end{equation*}
We observe that $\omega(\xx_{i + \frac 12}, \xx_i) = \left(2 \lambda_i\right)^{- \frac{2}{p-1}}$. Applying this fact and summing over all iterations $i$ yields 
\[
\sum_{i=0}^K  \lambda_i \frac{p!}{L_p} \langle F(\xx_{i+\frac{1}{2}}),\xx_{i+\frac 12 } -\xx\rangle \leq \omega(\xx, \xx_0) - \frac{15}{16} \sum_{i=0}^K \left(2 \lambda_i\right)^{- \frac{2}{p-1}}
\]
as desired. 
\end{proof}
We now state and prove our main theorem.
\begin{theorem}
Let $\epsilon >0$, $p \geq 1$ and $\calX\subseteq \mathbb{R}^n$ be any closed convex set. Let $\gg:\calX \rightarrow \mathbb{R}^n$ be an operator that is $p^{th}$-order $L_p$-smooth with respect to an arbitrary norm $\norm{\cdot}$. Let $\omega(\cdot,\cdot)$ denote the Bregman divergence of a function that is strongly convex with respect to the same norm $\norm{\cdot}$. Algorithm \ref{alg:MainAlgo} returns $\hat{\xx}$ such that $\forall \xx \in \calX$,
\[
\langle\gg(\xx),\hat{\xx}-\xx\rangle \leq \epsilon ,
\]
in at most 
\[
\frac{16}{15} \left(\frac{2L_p}{p!}\right)^{\nfrac{2}{p+1}}\frac{\omega(\xx,\xx_0)}{\epsilon^{\nfrac{2}{p+1}}}
\]
calls to an oracle that solves the subproblem defined in Definition \ref{def:VISub}.
\end{theorem}
\begin{proof}
Let $S_K = \sum_{i=0}^K \lambda_i$. We first note that, $\forall \xx \in \calX$,
\begin{align*}
\langle\gg(\xx),\hat{\xx}-\xx\rangle & = \sum_{i=0}^K \frac{\lambda_i}{S_K} \left\langle F(\xx),\xx_{i+\frac{1}{2}} -\xx \right\rangle\\
& \leq \sum_{i=0}^K \frac{\lambda_i}{S_K} \left\langle F(\xx_{i+\frac{1}{2}}),\xx_{i+\frac{1}{2}} -\xx \right\rangle, &&\text{(From monotonicity of $\gg$)}\\
& \leq \frac{L_p}{S_Kp!}\omega(\xx,\xx_0). &&\text{(From Lemma \ref{lem:Induction} and $\omega(\xx,\yy) \geq 0, \forall \xx,\yy$)}
\end{align*}
It is now sufficient to find a lower bound on $S_K$. We will use Lemma \ref{lem:Bullins} for $q=p-1, \xi_i = (2 \lambda_i)^{- \frac{1}{p-1}}$. Observe from Lemma \ref{lem:Induction} that 
\[
\sum_{i=1}^K \xi_i^2 = \sum_{i=0}^K (2 \lambda_i)^{- \frac{2}{p-1}}\leq \frac{16}{15} \omega(\xx,\xx_0) = \frac{16}{15} R^2.
\]
Now, Lemma \ref{lem:Bullins} gives
\[
2 S_K = 2 \sum_{i=0}^K\lambda_i = \sum_{i=0}^K\xi^{-(p-1)} \geq \frac{(K+1)^{\frac{q}{2} + 1}}{(\frac{16}{15} R^2)^{\frac{q}{2}}}.
\]
We thus have for all $\xx \in \calX$,
\[
\langle\gg(\xx),\hat{\xx}-\xx\rangle \leq \frac{2 L_p}{p!}\frac{(\frac{16}{15})^{\frac{p-1}{2}}\omega(\xx,\xx_0)^{\frac{p+1}{2}}}{(K+1)^{\frac{p+1}{2}}},
\]
which gives an $\epsilon$ approximate solution after $\frac{16}{15} \cdot \left(\frac{2L_p}{p! \epsilon}\right)^{\frac{2}{p+1}}\omega(\xx,\xx_0)$ iterations.
\end{proof}

\section{Lower Bound for Higher-Order Smooth Variational Inequalities}\label{sec:LowerBound}

In this section, we prove a lower bound for the monotone variational inequality problem, for $p^{th}$-order smooth monotone operators $\gg$, when finding an $\eps$-approximate MVI solution, i.e., finding $\zz^{\star}\in \calZ$ such that, for $\epsilon>0$ and closed convex set $\calZ \subseteq \mathbb{R}^n$,
\begin{equation}\label{eq:VIProb}
\langle \gg(\zz),\zz^{\star}-\zz\rangle \leq \epsilon, \quad \forall \zz \in \calZ.
\end{equation}
Our analysis and hard instances are inspired by the constructions of \cite{nesterov2021implementable} and \cite{ouyang2021lower}.

\paragraph{Oracle for Computing Iterates.} We define the following model for computing iterates. For a $p^{th}$-order smooth operator $\gg$, consider methods which at every iteration compute stationary points of the following family of higher-order tensor polynomial for some $\aa \in \mathbb{R}^p,\gamma \in \mathbb{R}, m >1$: 
\begin{equation}\label{eq:Polynomial}
\Phi_{\aa,\gamma,m}(\hh) = \sum_{i = 0}^{p-1} a_i \nabla^i \gg(\zz) [\hh]^{i+1} + \gamma \|\hh\|_2^m. 
\end{equation}
Let $\Gamma_{\zz,\gg}(\aa,\gamma,m)$ denote the set of all stationary points of the above polynomial. Define the linear subspace
\[
S_{\gg}(\zz) = span \{\Gamma_{\zz,\gg}(\aa,\gamma,m): \aa \in \mathbb{R}^p, \gamma>0,m>1\}.
\]
\begin{assumption}\label{ass:oracle}
For a $p^{th}$-order smooth operator $\gg$, we consider methods that generate a sequence of points $\{\zz_k\}_{k\geq 0} \in \calZ$ satisfying
\[
\zz^{(k+1)} \in \zz^{(0)} + \sum_{i=1}^k S_{\gg}(\zz^{(i)}).
\]
\end{assumption}

\paragraph{Hard Instance.}We will work with the following family of saddle point problems parameterized by $t\in \{1,2,\hdots n-1\}$,
\begin{equation}\label{eq:MinMax}
\min_{\xx\in \calX}\max_{\yy \in \calY} \zeta_t(\xx,\yy) = \ff_{t}(\xx) + \langle\AA_{t}\xx-\bb_t,\yy\rangle,
\end{equation}
for closed convex sets $\calX \subseteq \mathbb{R}^n$ and $\calY\subseteq \mathbb{R}^m$, $m\leq n$ and, $p^{th}$-order smooth, convex function $\ff_{t}$, matrix $\AA_{t}\in \mathbb{R}^{m\times n}$, vector $\bb_{t} \in \mathbb{R}^m$. We prove that these problems require at least $\approx t^{-\nfrac{(p+1)}{2}}$ iterations to converge.

Note that Problem \eqref{eq:MinMax} is a special case of Problem \eqref{eq:VIProb} for $\zz = (\xx,\yy)$, $\calZ = \calX \times \calY$, and
\begin{equation}\label{eq:Operator}
\gg = \begin{bmatrix}\nabla\ff_{t}+\AA_{t}^{\top}\yy\\ \AA_{t}\xx-\bb_t\end{bmatrix}.
\end{equation}
We now define the function $\ff_t$, matrix $\AA_t$ and vector $\bb_t$ similar to \cite{nesterov2021implementable} and \cite{ouyang2021lower}. For $t \in \{1,\dots n-1\}$,
\begin{equation*}
\ff_t(\xx) = \frac{L_{\ff}}{(p+1)!}\left(\sum_{i = 1}^{t} \abs{\BB_t\xx}_i^{p+1} + \sum_{i=t+1}^n \abs{\xx}_i^{p+1}\right) - \frac{1}{p!} \left(L_{\ff} + \frac{L_{\AA}}{2}\right)\xx\cdot \ee_{1,n}.
\end{equation*}
\[
\AA_t = \frac{L_{\AA}}{p!}\begin{bmatrix}
\BB_t & 0 \\
0 & \GG
\end{bmatrix}, \quad \bb_t = \frac{L_{\AA}}{p!} \begin{bmatrix} 1_{t}\\0 \end{bmatrix}.
\]
Here $L_{\AA} \geq 0, L_{\ff}>0$ and $L_{\ff}\geq L_{\AA}$. For $m < n$, $\AA\in \mathbb{R}^{m \times n}, \GG \in \mathbb{R}^{(m-t)\times (m-t)}$ is a full rank matrix s.t. $\|\GG\| = 2$, and $\BB_t \in \mathbb{R}^{t\times t}$ is defined as
\[
\BB_t = \begin{bmatrix}
& & &  & 1\\
& &  & 1 & -1\\
& & \iddots &\iddots & \\
& 1 & -1 & & \\
1 & -1 & & &
\end{bmatrix}.
\]
We note that $\ff_t$ is $L_{\ff}\cdot \|\BB_t\|^{p+1}\leq 2^{p+1}L_{\ff}$, $p^{th}$-order smooth and $\|\AA\| = \frac{2}{p!}L_{\AA}$.

Before we state our main result, we define sets $\calX,\calY$ and the primal and dual problems associated with Problem \eqref{eq:MinMax}.

\begin{equation}\label{eq:setsXY}
\calX = \{\xx\in \mathbb{R}^n: \|\xx\|_2^2 \leq \calR^2_{\calX}= 3(t+1)^3\}, \quad \calY = \{\yy\in \mathbb{R}^m: \|\yy\|_2^2 \leq \calR^2_{\calY}= t+1\}.
\end{equation}
The associated primal and dual problems are defined as,
\begin{equation}\label{eq:ProbPrimal}
\min_{\xx \in \calX} \quad  \phi_{t}(\xx) = \ff_{t}(\xx) + \max_{\yy \in \calY}\langle \AA_{t}\xx-\bb_{t},\yy\rangle 
\end{equation}
\begin{equation}
\max_{\yy \in \calY} \quad  \psi_{t}(\yy) = \langle \AA_{t}\xx-\bb_{t},\yy\rangle + \min_{\xx \in \calX} \ff_{t}(\xx).
\end{equation}

We are now ready to state our lower bound.
\begin{theorem}\label{thm:SPPLower}
Let $p\geq 2$, $1\leq t \leq \frac{n-1}{2}$, $L_{\ff}>0,L_{\AA}\geq 0$ and $L_{\ff}\geq L_{\AA}$. Let $(\bar{\xx},\bar{\yy})\in \calX \times \calY$ be the output after $t$ iterations of a method $\calM$ that satisfies Assumption \ref{ass:oracle}. when applied to Problem \ref{eq:MinMax} for $\zeta_{2t+1}$. Then,
\[
\phi_{2t+1}(\bar{\xx})- \psi_{2t+1}(\bar{\yy}) \geq \frac{1}{10\cdot 3^{\frac{3(p+1)}{2}}}\frac{pL_{\ff}}{(p+1)!} \frac{\calR_{\calX}^{p+1}}{(t+1)^{\frac{3p+1}{2}}} + \frac{L_{\AA}}{p!}\frac{\calR_{\calX}\calR_{\calY}^{p}}{\sqrt{3}(t+1)^{\frac{p+1}{2}}}.
\]
\end{theorem}

\subsection{A Lower Bound for Highly-Smooth Saddle-Point Problems}

We now work towards proving Theorem~\ref{thm:SPPLower}. We rely on the following lemmas, whose proofs can be found in Appendix \ref{app:lowerBoundProofs}. We begin by characterizing the iterates produced by a method $\calM$ satisfying Assumption \ref{ass:oracle}, when applied to the primal problem \eqref{eq:ProbPrimal}.

\begin{restatable}{lemma}{IteratesMinMax}\label{lem:iteratesMinMax}
Any method $\calM$ satisfying Assumption \ref{ass:oracle} applied to the Primal Problem \eqref{eq:ProbPrimal} for $\calX = \mathbb{R}^n$ and $\calY$ as defined in \eqref{eq:setsXY}, starting from $\xx^{(0)}=0$ generates points $\{\xx^{(k)}\}_{k \geq 0}$ satisfying
\[
\xx^{(k+1)} \in \sum_{i = 0}^k S_{\nabla \phi_t}(\xx^{(i)})\subseteq \mathbb{R}^n_{k+1} , \quad 0 \leq k \leq t-1.
\]
\end{restatable}

Next, we compute the values of the optimizer and the optimum of Problem \eqref{eq:MinMax}.

\begin{restatable}{lemma}{OptMinMax}\label{lem:OptMinMax}
For Problem \eqref{eq:MinMax} with $\calX,\calY$ as defined in \eqref{eq:setsXY},the optimal solution is given by

\[
(\xx_{2t+1})_i^{\star} = \begin{cases}
(2t+1) - i +1 & \text{ if } 1 \leq i \leq 2t+1,\\
0 & \text{ otherwise}.
\end{cases}, \quad \quad 
\yy_{2t+1}^{\star} = \frac{1}{2}
\begin{bmatrix}
1_{2t+1}\\
0
\end{bmatrix},
\]

 and the optimal objective value is 
\[
\zeta_{2t+1}^{\star} = - \frac{\frac{p}{p+1}L_{\ff}+ \frac{L_{\AA}}{2}}{p!}(2t+1).
\]
\end{restatable}

Our final lemma, before we prove our main result, bounds the minimum values of the function $\ff_{2t+1}$ and the norm $\|\AA_{2t+1}\xx -\bb_{2t+1}\|_2$, which we will need to prove the final bound.
\begin{restatable}{lemma}{OptFunctionOnly}\label{lem:FuncOnlyBound}
For $\ff_{2t+1}, \AA_{2t+1},\bb_{2t+1}$ as defined above, the following holds,
\begin{align*}
 \min_{\xx\in \mathbb{R}^n_t} \ff_{2t+1}\left(\xx\right) &\geq  \frac{ p L_{\ff}}{(p+1)!}\left(\frac{3}{2}\right)^{1+\frac{1}{p}} t, \text{ and,}\\
\min_{\xx\in \mathbb{R}^n_t} \|\AA_{2t+1}\xx -\bb_{2t+1}\|_2 &\geq \frac{L_{\AA}}{p!}(t+1).
\end{align*}
\end{restatable}

We are now ready to prove Theorem \ref{thm:SPPLower}.
 \subsubsection*{Proof of Theorem \ref{thm:SPPLower}}
\begin{proof}
We first claim that it is sufficient to lower bound $\min_{\xx\in  \mathbb{R}^n_t} \phi_{2t+1}(\xx) - \phi_{2t+1}^{\star}$. To see this, first note that since $\bar{\yy}\in\calY$, and $\psi_{2t+1}(\bar{\yy})$ is the dual objective, from weak duality,
\[
\psi_{2t+1}(\bar{\yy})\leq \psi_{2t+1}^{\star} \leq \phi_{2t+1}^{\star}.
\]
From Lemma \ref{lem:iteratesMinMax} after $t$ iterations all iterates produced by $\calM$ when applied to the problem $\min_{\xx\in \mathbb{R}^n}\phi_{2t+1}(\xx)$ must belong to the space $\mathbb{R}^n_t$. We now have the following,
\[
\phi_{2t+1}(\bar{\xx})- \psi_{2t+1}(\bar{\yy}) \geq \phi_{2t+1}(\bar{\xx}) -  \phi_{2t+1}^{\star} \geq  \min_{\xx\in  \mathbb{R}^n_t} \phi_{2t+1}(\xx) - \phi_{2t+1}^{\star},
\]
which proves our claim. In the remaining proof, we will focus on lower bounding $\min_{\xx\in  \mathbb{R}^n_t} \phi_{2t+1}(\xx) - \phi_{2t+1}^{\star}$.

Since $\calY$ is a Euclidean ball, 
\[
\max_{\yy \in \calY} \langle \AA_{2t+1}\xx- \bb_{2t+1},\yy\rangle = \calR_{\calY}\|\AA_{2t+1}\xx-\bb_{2t+1}\|_2,
\]
which gives us $\phi_{2t+1}(\xx) = \ff_{2t+1}(\xx) + \calR_{\calY}\|\AA_{2t+1}\xx-\bb_{2t+1}\|_2$.

\begin{align*}
\min_{ \xx\in \mathbb{R}^n_t} \phi(\xx) - \phi^{\star} & \geq \min_{\xx\in \mathbb{R}^n_t} \ff_{2t+1}(\xx) + \min_{\xx\in \mathbb{R}^n_t}\calR_{\calY}\|\AA_{2t+1}\xx-\bb_{2t+1}\|_2- \phi^{\star}\\
& \geq -\frac{pL_{\ff}}{(p+1)!}\left( \frac{3}{2}\right)^{1+\frac{1}{p}} t + \calR_{\calY} \frac{L_{\AA}}{p!}(t+1) +  \frac{\frac{p}{p+1}L_{\ff}+ \frac{L_{\AA}}{2}}{p!}(2t+1)\\
& \text{(Using the lower bound on the first two terms from lemma \ref{lem:FuncOnlyBound},}\\
& \text{and value of $\phi^{\star}$ from Lemma \ref{lem:OptMinMax})}\\
& = \frac{\left(\frac{p}{10(p+1)}L_{\ff}+ \frac{L_{\AA}}{2}\right)}{p!} (t+1) + \calR_{\calY} \frac{L_{\AA}}{p!}(t+1)\\
& \text{(Since for $p\geq 2$, $2-\left(1.5\right)^{1+\frac{1}{p}}\geq 2-1.5^{1.5}\geq 0.1$)}\\
& \geq \frac{pL_{\ff}}{10(p+1)!}(t+1) + \calR_{\calY}\frac{L_{\AA}}{p!}(t+1)\\
& \geq \frac{pL_{\ff}}{10\cdot 3^{\frac{3(p+1)}{2}}(p+1)!} \frac{\calR_{\calX}^{p+1}}{(t+1)^{\frac{3p+1}{2}}} + \frac{L_{\AA}}{p!}\frac{\calR_{\calX}\calR_{\calY}^{p}}{\sqrt{3}(t+1)^{\frac{p+1}{2}}}.
\end{align*}
The last inequality follows from the fact $\calR_{\calX} =\sqrt{3}(t+1)^{3/2}$ and $\calR_{\calY} = \sqrt{t+1}$. This concludes the proof of the theorem.
\end{proof}

\section{Conclusions}
In this paper, we have presented an algorithm for solving
$p^{\textrm{th}}$-order smooth MVI problems that converges at a
rate of $O(\epsilon^{-2/(p+1)}),$ without any line search as required
by previous methods. Our algorithm is simple and can be applied to
constrained and non-Euclidean settings. Our algorithm requires solving
an MVI subproblem in every iteration obtained by regularizing the
$p^{th}$-order Taylor expansion of the operator.

The MVI subproblems solved by our algorithm in each iteration are the
same as those arising in previous works, and when restricted to the
case of unconstrained convex optimization and Euclidean norms, they
become identical to those from optimal higher-order smooth convex
optimization algorithms. We further demonstrate in the appendix
that it is sufficient to solve these subproblems approximately, and
give an efficient algorithm for solving them for $p=2.$ Solving these
subproblems efficiently for $p \ge 3$ is an open problem even for the
special case of unconstrained convex optimization with Euclidean norms.

Finally, we provide a lower bound that matches the above rate up to
constant factors, thus showing that our algorithm is optimal. This
settles the oracle complexity of solving highly-smooth MVIs, and
establishes a gap between the rates achievable for highly-smooth
convex optimization and those for highly-smooth MVIs.

\pagebreak
\bibliographystyle{abbrvnat}
\bibliography{papers}

\appendix

\section{Proofs from Section \ref{sec:LowerBound}}\label{app:lowerBoundProofs}

\IteratesMinMax*

\begin{proof}
 We first prove that $\xx \in \mathbb{R}^n_{k}$ implies $S_{\nabla\phi_t}(\xx)\subseteq \mathbb{R}^n_{k+1}$. Since the space $S_{\nabla\phi_t}(\xx)$ is defined by the span of the stationary points of a polynomial defined by the directional derivatives of $\phi$, we first compute all directional derivatives. For simplicity of notation we let $\CC_t = \begin{bmatrix}\BB_t & 0 \\ 0 & \II_{n-t}\end{bmatrix}$, so that $\ff_t(\xx) = \frac{L_{\ff}}{(p+1)!} \|\CC_t\xx\|_{p+1}^{p+1} - \frac{1}{p!} \left(L_{\ff} + \frac{L_{\AA}}{2}\right)\xx\cdot \ee_{1,n}$. We can explicitly compute $\max_{\yy \in \calY} \langle \AA\xx-\bb,\yy\rangle = \calR_{\calY}\| \bb-\AA\xx\|_2.$ We thus have,
\begin{align*}
\nabla\phi_t(\xx)[\hh] & = \nabla\ff_t(\xx)^{\top}\hh + \calR_{\calY}\frac{\AA_t^{\top}\bb_t-\AA_t^{\top}\AA_t\xx}{\|\bb_t-\AA_t\xx\|_2}\cdot \hh\\
& = \frac{L_{\ff}}{p!}(\CC_t\xx)^{\top}\diag{|\CC_t\xx|^{p-1}}\CC_t\hh  + \frac{1}{p!} \left(L_{\ff} + \frac{L_{\AA}}{2}\right)\hh_1 - \calR_{\calY}\frac{\AA_t^{\top}\bb_t-\AA_t^{\top}\AA_t\xx}{\|\bb_t-\AA_t\xx\|_2}\cdot \hh.
\end{align*}

For $2\leq i\leq p-1$,
\[
\nabla^i \phi_t(\xx)[\hh]^{i} = \nabla^{i} \ff_t(\xx)[\hh]^{i} + \calR_{\calY}\cdot  \nabla^{i}\left(\|\bb_t-\AA_t\xx\|_2\right)[\hh]^i.
\]
From the proof of Lemma 2 of \cite{nesterov2021implementable},
\[
\nabla^j \ff_t(\xx)[\hh]^j = \sum_{i=1}^k \dd_{i,j}\langle e_{i,n}, \CC_t\hh\rangle^j, \quad 2 \leq j \leq p.
\]
Here $\dd_{i,j}$ are defined for $i = 1,\hdots n, j=1 \hdots p $ and is some scalar function of $\CC_t\xx$. We next compute  $\nabla^{i}\left(\|\bb_t-\AA_t\xx\|_2\right)[\hh]^i$.

Let $\hh(\vv) = \|\vv\|_2$ and $\vv(\xx) = \bb_t-\AA_t\xx$ so that $\|\bb_t-\AA_t\xx\|_2 = \hh\circ \vv(\xx)$. In order to compute these higher order directional derivatives, we will use Fa\`{a}
di Bruno's formula. Since $\nabla^i_{\xx}\vv = 0$ for $i\geq 2$ and $\AA_t$ for $i=1$, the higher order derivatives of our function are as,
\[
\nabla^{i}\left(\|\bb_t-\AA_t\xx\|_2\right)[\hh]^i = \left(\nabla_{\vv}^i\hh \circ \vv\right)(\nabla_{\xx}\vv)^{\otimes i}[h]^i.
\]

We can recursively define the derivatives as follows. For any $i\leq p-1$
\[
[\nabla_{\vv}^i\hh(\vv)]_{j_1\neq j_2\neq \dots\neq j_i} = (-1)^{i+1}\cdot \frac{\vv_{j_1}\vv_{j_2}\dots \vv_{j_i}}{\|\vv\|_2^{2i-1}}.
\]
\[
[\nabla_{\vv}^i\hh(\vv)]_{j_1\neq j_2\neq \dots\neq j_{i-1} = j_{i}} = (-1)^{i+1}\cdot \frac{\vv_{j_1}\vv_{j_2}\dots \vv_{j_{i-1}}}{\|\vv\|_2^{2i-1}} + (-1)^i\frac{\vv_{j_1}\vv_{j_2}\dots \vv_{j_i}\vv_{j_{i+1}}}{\|\vv\|_2^{2i+1}}.
\]
All other permutations of $j_{1},\dots j_{i}$ would give $[\nabla^i\hh(\vv)]_{j_{1},\dots j_{i}}$ that has a similar structure as above i.e., a multinomial expression of the coordinates of $\vv$. 
We can thus compute for $c_{i,j}$'s, $1\leq i\leq n, 1\leq j\leq p$ which are functions of $\AA_t^{\top}(\bb-\AA\xx)$,
\[
\nabla^{j}\left(\|\bb_t-\AA_t\xx\|_2\right)[\hh]^j = \sum_{i=1}^k c_{i,j}\langle \ee_{i,n},\hh\rangle^j.
\]
Here, the sum is only from $i=1$ to $k$ since if $\xx \in \mathbb{R}^n_{k}$ then $\AA^{\top}(\bb-\AA_t\xx) \in \mathbb{R}^n_{k}$.

The gradients of these derivatives with $\hh$ are,
\[
\nabla_{\hh}\nabla\phi_t(\xx)[\hh] = \frac{L_{\ff}}{p!}\CC_t^{\top}\diag{|\CC_t\xx|^{p-1}}\CC_t\xx  - \frac{1}{p!} \left(L_{\ff} + \frac{L_{\AA}}{\sqrt{2}}\right)\ee_{1,n} + \calR_{\calY}\frac{\AA_t^{\top}\bb_t- \AA_t^{\top}\AA_t\xx}{\|\bb_t-\AA_t\xx\|_2}\in \mathbb{R}^n_{k+1}.
\]
\[
\nabla_{\hh}\nabla_{\xx}^j \phi_t(\xx)[\hh]^j = \sum_{i=1}^k j c_{i,j}\langle e_{i,n},\hh\rangle^{j-1} \ee_{i,n}, \quad 2 \leq j \leq p.
\]
Since $\CC_t\xx,\AA_t\xx \in \mathbb{R}^n_{k}$, $\nabla_{\hh}\nabla_{\xx}^j \phi_t(\xx)[\hh]^j \in \mathbb{R}^n_{k+1}$. Since the regularizer in \eqref{eq:Polynomial} is in the euclidean norm, all the stationary points of this function belong to $\mathbb{R}^n_{k+1}$ and as a result $S_{\nabla\phi_t}(\xx)\subseteq \mathbb{R}^n_{k+1}$.

It remains to prove $\xx^{(k)}\in \mathbb{R}^n_k$ which we show by induction. For $k = 0$, $\xx^{(0)}= 0$,  
\[
\nabla_{\hh}\nabla_{\xx}\phi_t(\xx^{(0)}) = 
- \frac{1}{p!} \left(L_{\ff} + \frac{L_{\AA}}{\sqrt{2}}\right) \ee_{1,n} +\calR_{\calY} \frac{\AA_t^{\top}\bb_t}{\|\bb_t\|_2},
\]
and since for $\xx^{(0)}=0$, $c_{i,j}$'s are a function of $\AA^{\top}\bb \in \mathbb{R}^n_1$,
\[
\nabla_{\hh}\nabla_{\xx}^i \phi_t(\xx^{(0)})[\hh]^{i} = \text{constant} \cdot \hh_1^i \ee_{1,n}, 2\leq i \leq p-1.
\]
All the above derivatives are in $\mathbb{R}^n_1$ which gives us $\xx^{(1)} \in \mathbb{R}^n_1 $ by Assumption \ref{ass:oracle}. Now, assume $\xx^{(i)} \in \mathbb{R}^n_i$ for all $1 \leq i \leq k$. Since we have shown that $S_{\nabla\phi_t}(\xx^{(k)})\subseteq \mathbb{R}^n_{k+1}$, again from Assumption \ref{ass:oracle}, $\xx^{(k+1)} \in \mathbb{R}^n_{k+1}$.
\end{proof}

\OptMinMax*
\begin{proof}
The optimality condition is that there exist $\xx^{\star} \in \calX$ and $\yy^{\star}\in \calY$ such that, for all $\xx \in \calX$ and $\yy \in \calY$,
\[
\langle \nabla \ff_{2t+1}(\xx^{\star})+\AA_{2t+1}^{\top}\yy^{\star},\xx^{\star}-\xx\rangle \leq 0, \quad \langle \AA_{2t+1}\xx^{\star}-\bb_{2t+1},\yy^{\star}-\yy\rangle \leq 0.
\]
Since $\AA_{2t+1}\xx_{2t+1}^{\star} = \bb_{2t+1}$, the second condition is satisfied. We note that 
\[\nabla \ff_{2t+1}(\xx_{2t+1}^{\star}) = \begin{bmatrix}\frac{L_{\ff}}{p!} \BB^{\top} Diag (\abs{\BB\xx_{2t+1}^{\star}}^{p-1})\BB\xx_{2t+1}^{\star} - \frac{1}{p!} \left(L_{\ff} + \frac{L_{\AA}}{2}\right)\ee_{1,2t+1}\\
\frac{L_{\ff}}{p!} |\xx_{2t+1}^{\star}|^{p-1}\xx_{2t+1}^{\star}
\end{bmatrix} = -\AA_{2t+1}^{\top}\yy_{2t+1}^{\star}.
\]
Therefore, the first condition also holds and $\xx_{2t+1}^{\star} \in \calX,\yy_{2t+1}^{\star}\in \calY$ is an optimizer. Evaluating the function value at this point gives us the value of $\zeta^{\star}$.
\end{proof}

\OptFunctionOnly*
\begin{proof}
Since $\xx \in \mathbb{R}^n_t$, from the definition of $\ff_{2t+1}$, we have $\ff_t \equiv \ff_{2t+1}$. Therefore, it is sufficient to look at the optimizer of $\min_{\xx \in \mathbb{R}^n_t}\ff_t(\xx)$. Let $\xx = \left(\uu^{\top}, \vv^{\top}\right)^{\top}$, $\uu \in \mathbb{R}^{t}, \vv \in \mathbb{R}^{n-t}$. The KKT condition is, $\nabla \ff_t(\xx) = 0$, i.e.,
\[
\frac{L_{\ff}}{p!} \BB^{\top} Diag (\abs{\BB\uu^{\star}}^{p-1})\BB\uu^{\star} - \frac{1}{p!} \left(L_{\ff} + \frac{L_{\AA}}{2}\right)\ee_{1,t} =0,
\]
and,
\[
\frac{L_{\ff}}{p!}  Diag (\abs{\vv^{\star}}^{p-1})\vv^{\star} =0.
\]
We thus have $\vv^{\star}$ = 0, and,
\[
L_{\ff}  \abs{\BB\uu^{\star}}^{p} sign(\BB\uu^{\star}) =  \left(L_{\ff} + \frac{L_{\AA}}{2}\right)1_t,
\]
or,
\[
\BB\uu^{\star} = \left(1 + \frac{L_{\AA}}{2 L_{\ff}}\right)^{1/p}1_t, \quad \uu_1 = \left(1 + \frac{L_{\AA}}{2 L_{\ff}}\right)^{1/p} \cdot t .
\]
Plugging these values back in the main objective gives,
\begin{align*}
\ff_t^{\star} & = \frac{L_{\ff}}{(p+1)!}\left(1 + \frac{L_{\AA}}{2L_{\ff}}\right)^{1+\frac{1}{p}} t - \frac{1}{p!}\left(L_{\ff} + \frac{L_{\AA}}{2}\right)\left(1 + \frac{L_{\AA}}{2L_{\ff}}\right)^{\frac{1}{p}} \cdot t\\
& = - \frac{p L_{\ff}}{(p+1)!}\left(1 + \frac{L_{\AA}}{2 L_{\ff}}\right)^{1+\frac{1}{p}} t
\end{align*}
Since $L_{\ff}\geq L_{\AA}$, the above reduces to,
\[
\ff_t^{\star} \geq - \frac{ p L_{\ff}}{(p+1)!}\left(1 + \frac{1}{2}\right)^{1+\frac{1}{p}} t
\]
We next bound $\min_{\xx\in \mathbb{R}^n_t} \|\AA_{2t+1}\xx -\bb_{2t+1}\|_2$.

Since for any $\xx\in \mathbb{R}_t^n$, only the first $t$ entries can be non-zero, $(\AA_{2t+1}\xx-\bb_{2t+1})_i = (\bb_{2t+1})_i = 1, $ for $ i \in [t+1,2t+1]$. We thus have,
\begin{align*}
   \min_{\xx\in \mathbb{R}^n_t}  \|\AA_{2t+1}\xx - \bb_{2t+1}\|_2 &\geq \frac{L_{\AA}}{p!}\sqrt{t+1}\\
   & \geq \frac{L_{\AA}}{p!}\frac{\sqrt{t+1}\|\xx_{2t+1}^{\star}\|_2\|\yy_{2t+1}^{\star}\|_2^{p-1}}{\sqrt{3}(t+1)^{\frac{p+2}{2}}} \\
   &= \frac{L_{\AA}}{p!}\frac{\|\xx_{2t+1}^{\star}\|_2\|\yy_{2t+1}^{\star}\|_2^{p-1}}{\sqrt{3}(t+1)^{\frac{p+1}{2}}}.
\end{align*}
\end{proof}

\section{Approximate MVI Solution}
We now show we may handle approximation errors within the standard VI analysis.

\begin{algorithm}
\caption{Algorithm for Higher-Order Smooth MVI Optimization (Approximate Subproblem Solve)}\label{alg:MainAlgoApprox}
 \begin{algorithmic}[1]
\Procedure{MVI-OPT-APPROX}{$\xx_0 \in \calX, K, p, \delta$}
\For{ $i=0$ to $i=K$}
  \State $\xx_{i+\frac{1}{2}}  \leftarrow \textsc{Approx-VI-Solve}_{p,\delta}(\xx_i)$
  \State $\lambda_i \leftarrow \frac{1}{2}\omega(\xx_{i+\frac{1}{2}},\xx_i)^{-\frac{p-1}{2}}$
  \State $\xx_{i+1} \leftarrow \arg\min_{\xx \in \calX} \left\{\langle F(\xx_{i+\frac{1}{2}}),\xx -\xx_{i+\frac{1}{2}}\rangle  + \frac{L_p}{p!\lambda_i} \omega(\xx,\xx_i)\right\}$
\EndFor
\State\Return $\hat{\xx}_K = \frac{\sum_{i=0}^K \lambda_i \xx_{i+\frac{1}{2}}}{\sum_{i=0}^K \lambda_i}$
\EndProcedure 
 \end{algorithmic}
\end{algorithm}

To begin, we need to establish a variant of Lemma \ref{lem:Induction} that is specific to the case where we only have an approximate solution. Note that the proof remains nearly the same as before.

\begin{lemma}\label{lem:InductionSO}
Suppose, for any $\bar{\xx} \in \calX$, $\textsc{Approx-VI-Solve}_{p,\delta}(\bar{\xx})$ outputs a $\delta$-approximate solution to the regularized $p^{th}$-order MVI given in Definition \ref{def:VISub}. Then, for any $K \geq 1$ and $\xx \in \calX$, the iterates $\xx_{i + \frac 12}$ and parameters $\lambda_i$ in Algorithm \ref{alg:MainAlgoApprox} satisfy
\[
 \frac{p!}{L_p}\sum_{i=0}^K  \left(\lambda_i \langle F(\xx_{i+\frac{1}{2}}),\xx_{i+\frac 12 } -\xx\rangle -\delta\right) \leq \omega(\xx, \xx_0) - \frac{15}{16} \sum_{i=0}^K \left(2 \lambda_i\right)^{- \frac{2}{p-1}}.
\]
\end{lemma}
\begin{proof}
For any $i$ and any $\xx \in \calX$, we first apply Lemma~\ref{lem:Tseng} with $\phi(\xx) = \lambda_i \frac{p!}{L_p} \langle F(\xx_{i+\frac{1}{2}}),\xx - \xx_i\rangle$, which gives us
\begin{equation}\label{eq:Tseng}
\lambda_i \frac{p!}{L_p} \langle F(\xx_{i+\frac{1}{2}}),\xx_{i+1} -\xx\rangle \leq \omega(\xx,\xx_i) - \omega(\xx,\xx_{i+1}) - \omega(\xx_{i+1},\xx_i). 
\end{equation}
Additionally, by assumption, the guarantee of the output of \textsc{Approx-VI-Solve} is such that
\begin{equation}
\label{eqn:MPstep1}
\left\langle \mathcal{T}_{p-1}(\xx_{i+\frac 12}; \xx_i) , \xx_{i+\frac 12} - \xx_{i+1} \right\rangle  \leq \frac{2 L_p}{p!} \omega(\xx_{i+\frac 12}, \xx_i)^{\frac{p-1}{2}} \left\langle \nabla \dd(\xx_i) - \nabla \dd(\xx_{i+\frac 12}),  \xx_{i+\frac 12} - \xx_{i+1} \right\rangle + \delta.
\end{equation}
Applying the Bregman three point property (Lemma~\ref{lem:3point}) and the definition of $\lambda_k$ to Equation~\ref{eqn:MPstep1}, we have
\begin{equation}
\label{eqn:Tseng2}
\lambda_i \frac{p!}{L_p}  \left\langle \mathcal{T}_{p-1}(\xx_{i+\frac 12}; \xx_i) , \xx_{i+\frac 12} - \xx_{i+1} \right\rangle  \leq \omega(\xx_{i+1}, \xx_i) - \omega(\xx_{i+1}, \xx_{i+\frac 12}) - \omega(\xx_{i+\frac 12}, \xx_i) + \lambda_i \frac{p!}{L_p}\delta
\end{equation}

Summing Equations~\ref{eq:Tseng} and \ref{eqn:Tseng2}, we obtain 
\begin{align}
  \label{eqn:one_iter_1}
&\lambda_i \frac{p!}{L_p} \left( \langle F(\xx_{i+\frac{1}{2}}),\xx_{i+\frac 12 } -\xx\rangle + \left\langle   \mathcal{T}_{p-1}(\xx_{i+\frac 12}; \xx_i) - F(\xx_{i + \frac 12 }) , \xx_{i+\frac 12} - \xx_{i+1} \right\rangle \right) \nonumber \\
&\leq \omega(\xx,\xx_i) - \omega(\xx,\xx_{i+1}) - \omega(\xx_{i+1}, \xx_{i+\frac 12}) - \omega(\xx_{i+\frac 12}, \xx_i) + \lambda_i \frac{p!}{L_p} \delta.
\end{align}

Now, we obtain 
\begin{align*}
\lambda_i \frac{p!}{L_p} &\left\langle   \mathcal{T}_{p-1}(\xx_{i+\frac 12}; \xx_i) - F(\xx_{i + \frac 12 }) , \xx_{i+\frac 12} - \xx_{i+1} \right\rangle \\&\substack{(i) \geq} - \lambda_i \frac{p!}{L_p} \norm{\mathcal{T}_{p-1}(\xx_{i+\frac 12}; \xx_i) - F(\xx_{i + \frac 12 }) }_* \norm{ \xx_{i+\frac 12} - \xx_{i+1}} \\
&\substack{(ii) \\ \geq} - \lambda_i \norm{\xx_{i+ \frac 12} -  \xx_i}^p \norm{ \xx_{i+\frac 12} - \xx_{i+1}} \\ 
&\substack{(iii) \\ \geq} -\frac{1}{2} \omega(\xx_{i + \frac 12}, \xx_i)^{-\frac{p-1}{2}} \omega(\xx_{i + \frac 12}, \xx_i)^{\frac{p}{2}} \omega(\xx_{i+1}, \xx_{i+\frac 12})^{\frac{1}{2}} \\ 
&= - \frac{1}{2} \omega(\xx_{i + \frac 12}, \xx_i)^{\frac{1}{2}} \omega(\xx_{i+1}, \xx_{i+\frac 12})^{\frac{1}{2}} \\
&\substack{(iv) \\ \geq} - \frac{1}{16}  \omega(\xx_{i + \frac 12}, \xx_i) - \omega(\xx_{i+1}, \xx_{i+\frac 12}).
\end{align*}
Here, $(i)$ used H\"older's inequality, $(ii)$ used  Definition~\ref{def:Lipschitz}, $(iii)$ used the $1$-strong convexity of $\omega$, and $(iv)$ used the inequality $\sqrt{xy} \leq 2x + \frac{1}{8} y$ for $x,y \geq 0$. Combining with \ref{eqn:one_iter_1} and rearranging yields
\begin{equation}
\lambda_i \frac{p!}{L_p} \langle F(\xx_{i+\frac{1}{2}}),\xx_{i+\frac 12 } -\xx\rangle \leq \omega(\xx,\xx_i) - \omega(\xx,\xx_{i+1}) - \frac{15}{16} \omega(\xx_{i+\frac 12}, \xx_i) + \lambda_i \frac{p!}{L_p}\delta. 
\end{equation}
We observe that $\omega(\xx_{i + \frac 12}, \xx_i) = \left(2 \lambda_i\right)^{- \frac{2}{p-1}}$. Applying this fact and summing over all iterations $i$ yields 
\[
\sum_{i=0}^K  \lambda_i \frac{p!}{L_p} \langle F(\xx_{i+\frac{1}{2}}),\xx_{i+\frac 12 } -\xx\rangle - \frac{p!}{L_p} \delta \sum_{i=1}^K \lambda_i\leq \omega(\xx, \xx_0) - \frac{15}{16} \sum_{i=0}^K \left(2 \lambda_i\right)^{- \frac{2}{p-1}}   ,
\]
as desired. 
\end{proof}

We now state and prove the main theorem of this section.
\begin{theorem}\label{thm:mainApproxThm}
Let $\epsilon >0$, $p \geq 1$, $\delta \leq \frac{ \epsilon}{2}$, and let $\calX\subseteq \mathbb{R}^n$ be any closed convex set. Let $\gg:\calX \rightarrow \mathbb{R}^n$ be an operator that is $p^{th}$-order $L_p$-smooth with respect to an arbitrary norm $\norm{\cdot}$. Let $\omega(\cdot,\cdot)$ denote the Bregman divergence of a function that is strongly convex with respect to the same norm $\norm{\cdot}$. Algorithm \ref{alg:MainAlgoApprox} returns $\hat{\xx}$ such that $\forall \xx \in \calX$,
\[
\langle\gg(\xx),\hat{\xx}-\xx\rangle \leq \epsilon ,
\]
in at most 
\[
\frac{16}{15} \left(\frac{4L_p}{p!}\right)^{\nfrac{2}{p+1}}\frac{\omega(\xx,\xx_0)}{\epsilon^{\nfrac{2}{(p+1)}}}
\]
calls to an \textsc{Approx-VI-Solve} subroutine.
\end{theorem}
\begin{proof}
Let $S_K = \sum_{i=0}^K \lambda_i$. We first note that, $\forall \xx \in \calX$
\begin{align*}
\langle\gg(\xx),\hat{\xx}-\xx\rangle - \delta& = \sum_{i=0}^K \frac{\lambda_i}{S_K} \left\langle F(\xx),\xx_{i+\frac{1}{2}} -\xx \right\rangle - \delta\\
& \leq \sum_{i=0}^K \frac{\lambda_i}{S_K} \left\langle F(\xx_{i+\frac{1}{2}}),\xx_{i+\frac{1}{2}} -\xx \right\rangle - \delta, &&\text{(From monotonicity of $\gg$)}\\
& \leq \frac{L_p}{S_Kp!}\omega(\xx,\xx_0), &&\text{(From Lemma \ref{lem:InductionSO} and $\lambda_i \geq 0, \forall i$)}
\end{align*}
It is now sufficient to find a lower bound on $S_K$. We will use Lemma \ref{lem:Bullins} for $q=p-1, \xi_i = (2 \lambda_i)^{- \frac{1}{p-1}}$. Observe from Lemma \ref{lem:InductionSO} that 
\[
\sum_{i=1}^K \xi_i^2 = \sum_{i=0}^K (2 \lambda_i)^{- \frac{2}{p-1}}\leq \frac{16}{15} \omega(\xx,\xx_0) = \frac{16}{15} R^2.
\]
Now, Lemma \ref{lem:Bullins} would give
\[
2 S_K = 2 \sum_{i=0}^K\lambda_i = \sum_{i=0}^K\xi^{-(p-1)} \geq \frac{(K+1)^{\frac{q}{2} + 1}}{(\frac{16}{15} R^2)^{\frac{q}{2}}}
\]
We thus have for all $\xx \in \calX$,
\[
\langle\gg(\xx),\hat{\xx}-\xx\rangle - \delta \leq \frac{2 L_p}{p!}\frac{(\frac{16}{15})^{\frac{p-1}{2}}\omega(\xx,\xx_0)^{\frac{p+1}{2}}}{(K+1)^{\frac{p+1}{2}}},
\]
which gives an $\epsilon$ approximate solution after $\frac{16}{15} \cdot \left(\frac{4L_p}{p! \epsilon}\right)^{\frac{2}{p+1}}\omega(\xx,\xx_0)$ iterations.
\end{proof}

\section{Solving the Subproblem for $p=2$}\label{sec:SubproblemSolve}
Following along the lines of previous work on solutions to trust region/cubic regularization subproblems \citep{nesterovpolyak2006cubic, carmon2020acceleration}, we now show how our VI subproblem may be approximately solved in the unconstrained Euclidean setting for $p=2$. Thus, in the case where $\calX = \rea^n$ and $d(\xx) = \norm{\xx}^2$, we have
\begin{align*}
U_{2,\xx}(\yy) &= \mathcal{T}_{1}(\yy;\xx) + 2 L_2 \norm{\yy-\xx}\left(\yy - \xx \right)\\
&= F(\xx) + \nabla F(\xx)(\yy - \xx) + 2L_2 \norm{\yy-\xx}\left(\yy - \xx \right),
\end{align*}
and so for any $\hat{\xx} \in \rea^n$, our subproblem is to find $T(\hat{\xx}) \in \rea^n$ such that
\begin{equation}\label{eq:cubicVI}
\langle F(\hat{\xx}) + \nabla F(\hat{\xx})(T(\hat{\xx}) - \hat{\xx}) + 2 L_2 \norm{T(\hat{\xx})-\hat{\xx}}\left(T(\hat{\xx}) - \hat{\xx} \right), T(\hat{\xx})-\xx\rangle \leq 0, \quad \forall \xx \in \rea^n.
\end{equation}

To begin, we characterize the solution to this VI via the following lemma.
\begin{lemma}
There exists a unique $\lambda^* \geq 0$ such that $T(\hat{\xx}) = \hat{\xx} - (\nabla F(\hat{\xx})+\lambda^*\bI)^{-1}F(\hat{\xx})$ is a solution to \eqref{eq:cubicVI}. Furthermore, $\frac{\lambda^*}{3L_2} = \norm{(\nabla F(\hat{\xx})+\lambda^*\bI)^{-1}F(\hat{\xx})}$.
\end{lemma}
\begin{proof}
The lemma follows from KKT optimality conditions. Let $\hat{\xx} \in \rea^n$, and consider the auxiliary functions 
\begin{align*}
\Phi(\yy, \lambda) &\defeq \left[F(\hat{\xx}) + \nabla F(\hat{\xx})(\yy - \hat{\xx}) + \frac{2}{3}\lambda\left(\yy - \hat{\xx} \right), \frac{1}{3}\norm{\yy-\hat{\xx}}^2 \right]^\top
\end{align*}
and $h(\yy,\lambda) \defeq \frac{9}{2}L_2^2\norm{\yy - \hat{\xx}}^2 - \frac{1}{2}\lambda^2$.
Note that a solution to 
\begin{align*}
\text{Find }\ (\yy^*, \lambda^*):\ \langle \Phi(\yy^*, \lambda^*), (\yy^*, \lambda^*)-(\yy, \lambda)\rangle \leq 0, \quad \forall (\yy, \lambda) \in \calY,
\end{align*}
for $\calY \defeq \left\{(\yy, \lambda) \in \rea^{n+1}\ |\ h(\yy,\lambda) = 0 \right\}$, gives a solution to \eqref{eq:cubicVI}.

By KKT optimiality conditions, we have that $(\yy^*, \lambda^*)$ is a solution when:
\begin{align*}
\Phi(\yy^*, \lambda^*) + \nabla h(\yy^*, \lambda^*)\nu^* &= 0\\
h(\yy^*, \lambda^*) &= 0,
\end{align*}
for some Lagrange multiplier $\nu^*$.
Equivalently, we have
\begin{align*}
F(\hat{\xx}) + \nabla F(\hat{\xx})(\yy^* - \hat{\xx}) + \frac{2}{3} \lambda^*\left(\yy - \hat{\xx} \right) + 9L_2^2\nu^*(\yy^* - \hat{\xx}) &= 0\\
\frac{1}{3}\norm{\yy^*-\hat{\xx}}^2 - \nu^*\lambda^* &= 0\\
\frac{9}{2}L_2^2\norm{\yy^* - \hat{\xx}}^2 - \frac{1}{2}\lambda^{*2} &= 0,
\end{align*}
Combining the last two equations gives us that $\nu^* = \frac{\lambda^*}{27L_2^2}$, and so we may equivalently rewrite the system as:
\begin{align*}
F(\hat{\xx}) + (\nabla F(\hat{\xx})+ \lambda^* \bI)(\yy^* - \hat{\xx}) &= 0\\
\frac{9}{2}L_2^2\norm{\yy^* - \hat{\xx}}^2 - \frac{1}{2}\lambda^{*2} &= 0.
\end{align*}
Finally, solving for the first equation gives $(\yy^* - \hat{\xx}) = -(\nabla F(\hat{\xx})+\lambda^*\bI)^{-1}F(\hat{\xx})$, and so $\yy^* = \hat{\xx} - (\nabla F(\hat{\xx})+\lambda^*\bI)^{-1}F(\hat{\xx})$.
\end{proof}

We now want to establish how closely we need to approximate $\lambda^*$ for a sufficiently accurate solution.
\begin{lemma}\label{lem:lambdaClose}
Let $\lambda^* \geq 0$ be such that $T(\hat{\xx}) =\hat{\xx} -(\nabla F(\hat{\xx})+\lambda^*\bI)^{-1}F(\hat{\xx})$ is a solution to \eqref{eq:cubicVI}, and suppose that, for $\mu > 0$, for all $\xx \in \rea^n$, $\xx^\top \nabla F(\hat{\xx}) \xx \geq \mu$. Then, for any $\lambda$ such that $\abs{\lambda - \lambda^*} \leq \frac{\delta\mu^2}{\norm{F(\hat{\xx})}}$, we have that
\begin{equation*}
\norm{(\nabla F(\hat{\xx})+\lambda\bI)^{-1}F(\hat{\xx}) - (\nabla F(\hat{\xx})+\lambda^*\bI)^{-1}F(\hat{\xx})} \leq \delta.
\end{equation*}
\end{lemma}
\begin{proof}
Let $\lambda > 0$. We first note that
\begin{align*}
&\norm{(\nabla F(\hat{\xx})+\lambda\bI)^{-1}F(\hat{\xx}) - (\nabla F(\hat{\xx})+\lambda^*\bI)^{-1}F(\hat{\xx})} \\
&= \norm{\left((\nabla F(\hat{\xx})+\lambda\bI)^{-1} - (\nabla F(\hat{\xx})+\lambda^*\bI)^{-1}\right)F(\hat{\xx})}\\
&= \norm{\left((\nabla F(\hat{\xx}) + \lambda\bI)^{-1} - (\nabla F(\hat{\xx}) +\lambda\bI - \lambda\bI +\lambda^*\bI)^{-1}\right)F(\hat{\xx})}\\
&= \Big\lVert\left((\nabla F(\hat{\xx})+\lambda\bI)^{-1} - (\nabla F(\hat{\xx})+\lambda\bI)^{-1}\left(\frac{1}{\lambda-\lambda^*}\bI+(\nabla F(\hat{\xx})+\lambda\bI)^{-1}\right)^{-1}(\nabla F(\hat{\xx})+\lambda\bI)^{-1}\right) F(\hat{\xx})\\
&\qquad \qquad - (\nabla F(\hat{\xx})+\lambda\bI)^{-1}F(\hat{\xx})\Big\rVert\\
&=\norm{(\nabla F(\hat{\xx})+\lambda\bI)^{-1}\left(\frac{1}{\lambda-\lambda^*}\bI+(\nabla F(\hat{\xx})+\lambda\bI)^{-1}\right)^{-1}(\nabla F(\hat{\xx})+\lambda\bI)^{-1}F(\hat{\xx})}\\
&\leq \abs{\lambda - \lambda^*}\norm{(\nabla F(\hat{\xx})+\lambda\bI)^{-1}}^2\norm{F(\hat{\xx})}\\
&\leq \delta,
\end{align*}
where the final inequality follows from the bound on $\abs{\lambda - \lambda^*}$.
\end{proof}

\begin{algorithm}
\caption{Approximate Solver for Second-Order MVI Subproblem}\label{alg:SubproblemSolve}
 \begin{algorithmic}[1]
\Procedure{Approx-SO-VI-Solve}{$\hat{\xx} \in \rea^n$, $\delta \in (0,1)$}
\State $l = 0$, $u = \frac{\norm{F(\hat{\xx})}}{\delta}$, $\nu = \frac{\delta\mu^2}{\norm{F(\hat{\xx})}}$, $\lambda = \frac{l+u}{2}$, $\lambda^{-} = \lambda - \nu$

\While{ \textbf{not} $\norm{(\nabla F(\hat{\xx})+\lambda\bI)^{-1}F(\hat{\xx})} \leq \lambda$ \textbf{and} $\norm{(\nabla F(\hat{\xx})+\lambda^{-}\bI)^{-1}F(\hat{\xx})} > \lambda^{-}$}
\If{ $\lambda \leq \frac{\delta\mu^2}{\norm{F(\hat{\xx})}}$ }
\State Break
\EndIf
\If{ $\norm{(\nabla F(\hat{\xx})+\lambda\bI)^{-1}F(\hat{\xx})} \leq \lambda$}
\State $u = \lambda$, $\lambda = \frac{l + u}{2}$, $\lambda^- = \lambda - \nu$
\Else
\State $l = \lambda$, $\lambda = \frac{l + u}{2}$, $\lambda^- = \lambda - \nu$
\EndIf
\EndWhile
\State\Return $\hat{\xx} -(\nabla F(\hat{\xx})+\lambda\bI)^{-1}F(\hat{\xx})$
\EndProcedure 
 \end{algorithmic}
\end{algorithm}

Next we want to ensure that our subproblem solver routine \textsc{Approx-SO-VI-Solve} (Algorithm \ref{alg:SubproblemSolve}) can find a solution that approximates the exact solution to sufficient accuracy.

\begin{theorem}\label{thm:approxSOerror}
Let $\delta \in (0,1)$. The output of \textsc{Approx-SO-VI-Solve} (Algorithm \ref{alg:SubproblemSolve}) given as $\tilde{T}(\hat{\xx}) = \hat{\xx} -(\nabla F(\hat{\xx})+\lambda\bI)^{-1}F(\hat{\xx})$ is such that
\begin{equation}
\langle U_{2,\hat{\xx}}(\tilde{T}(\hat{\xx})), \tilde{T}(\hat{\xx})-\xx\rangle \leq \delta\left(\frac{L_2}{2} + \norm{\nabla U_{2,\hat{\xx}}(T(\hat{\xx}))}\right)\norm{\tilde{T}(\hat{\xx})-\xx}, \quad \forall \xx \in \rea^n.
\end{equation}

In addition, the total computational cost is at most the cost of a single Schur decomposition, which takes $n^\omega$ time, where $\omega \approx 2.3728$ is the matrix multiplication constant, plus $O\left(\log\left(\frac{\norm{F\left(\hat{\xx}\right)}}{\mu\delta}\right)\right)$ calls to a linear system solver in a quasi-upper-triangular system, each of which requires $O(n)$ time.
\end{theorem}

\begin{proof}
Note that, by monotonicity of $\norm{(\nabla F(\hat{\xx})+\lambda\bI)^{-1}F(\hat{\xx})}$ in $\lambda$, along with uniqueness of $\lambda^*$, if it is the case that the conditions of the while loop in Algorithm \ref{alg:SubproblemSolve} are not met (and so we break), then we know that $\lambda^{-} \leq \lambda^* \leq \lambda$. Thus, since $\abs{\lambda - \lambda^{-}} \leq \frac{\delta\mu^2}{\norm{F(\hat{\xx})}}$, it follows that $\abs{\lambda - \lambda^{*}} \leq \frac{\delta\mu^2}{\norm{F(\hat{\xx})}}$.  If, on the other hand, we break out of the while loop due to $\lambda \leq \frac{\delta\mu^2}{\norm{F(\hat{\xx})}}$ (which will happen after at most $O\left(\log\left(\frac{\norm{F\left(\hat{\xx}\right)}}{\mu\delta}\right)\right)$ iterations of the loop), we know that $\abs{\lambda - \lambda^*} \leq \frac{\delta\mu^2}{\norm{F(\hat{\xx})}}$. Furthermore, we may precompute a Schur decomposition of $\nabla F(\hat{x}) = QUQ^{-1}$, whereby $U$ is quasi-upper-triangular (since $\nabla F(\hat{x})$ has all real entries), which means that $U$ is a block diagonal matrix with block size at most $2\times 2$. It follows that, for any $\lambda$, solving a system in $\nabla F(\hat{x}) + \lambda I = Q(U+\lambda I)Q^{-1}$ can be done in $O(n)$ time, and so the total computational cost will be at most $n^\omega$ + $O\left(n\log\left(\frac{\norm{F\left(\hat{\xx}\right)}}{\mu\delta}\right)\right)$.
 Now, by Lemma \ref{lem:lambdaClose} we know that \textsc{Approx-SO-VI-Solve} outputs $\tilde{T}(\hat{\xx}) = \hat{\xx} -(\nabla F(\hat{\xx})+\lambda\bI)^{-1}F(\hat{\xx})$ such that 
\begin{equation*}
\norm{\tilde{T}(\hat{\xx}) - T(\hat{\xx})} \leq \delta,
\end{equation*}
where we let $T(\hat{\xx}) =\hat{\xx} - (\nabla F(\hat{\xx})+\lambda^*\bI)^{-1}F(\hat{\xx})$. By optimality conditions for this unconstrained problem, we know that $U_{2,\hat{\xx}}(T(\hat{\xx})) = 0$. We now note that, for all $\xx \in \rea^n$,
\begin{align*}
\langle &U_{2,\hat{\xx}}(\tilde{T}(\hat{\xx})), \tilde{T}(\hat{\xx})-\xx\rangle  = \langle U_{2,\hat{\xx}}(\tilde{T}(\hat{\xx})) - U_{2,\hat{\xx}}(T(\hat{\xx})), \tilde{T}(\hat{\xx})-\xx\rangle\\
&\leq \norm{U_{2,\hat{\xx}}(\tilde{T}(\hat{\xx})) - U_{2,\hat{\xx}}(T(\hat{\xx}))}\norm{\tilde{T}(\hat{\xx})-\xx}\\
&= \Big\lVert U_{2,\hat{\xx}}(\tilde{T}(\hat{\xx})) - U_{2,\hat{\xx}}(T(\hat{\xx})) - \nabla U_{2,\hat{\xx}}(T(\hat{\xx}))(\tilde{T}(\hat{\xx}) - T(\hat{\xx})) \\
&\qquad \qquad \qquad + \nabla U_{2,\hat{\xx}}(T(\hat{\xx}))(\tilde{T}(\hat{\xx}) - T(\hat{\xx}))\Big\rVert\norm{\tilde{T}(\hat{\xx})-\xx}\\
&\leq \Big(\norm{U_{2,\hat{\xx}}(\tilde{T}(\hat{\xx})) - U_{2,\hat{\xx}}(T(\hat{\xx})) - \nabla U_{2,\hat{\xx}}(T(\hat{\xx}))(\tilde{T}(\hat{\xx}) - T(\hat{\xx}))} \\
&\qquad \qquad \qquad + \norm{\nabla U_{2,\hat{\xx}}(T(\hat{\xx}))(\tilde{T}(\hat{\xx}) - T(\hat{\xx}))}\Big)\norm{\tilde{T}(\hat{\xx})-\xx}\\
&\leq \left(\frac{L_2}{2}\norm{\tilde{T}(\hat{\xx}) - T(\hat{\xx})}^2 + \norm{\nabla U_{2,\hat{\xx}}(T(\hat{\xx}))}\norm{\tilde{T}(\hat{\xx}) - T(\hat{\xx})}\right)\norm{\tilde{T}(\hat{\xx})-\xx}\\
&\leq \left(\frac{L_2}{2}\delta^2 + \norm{\nabla U_{2,\hat{\xx}}(T(\hat{\xx}))}\delta\right)\norm{\tilde{T}(\hat{\xx})-\xx}\\
&\leq \delta\left(\frac{L_2}{2} + \norm{\nabla U_{2,\hat{\xx}}(T(\hat{\xx}))}\right)\norm{\tilde{T}(\hat{\xx})-\xx},
\end{align*}
which completes the proof.
\end{proof}

Now that we have established all of the prerequisite results, we may state and prove our main theorem concerning how to instantiate our method for the unconstrained Euclidean case, for $p=2$.
\begin{theorem}
Let $\epsilon >0$, and let $\calX = \mathbb{R}^n$. Let $\gg:\calX \rightarrow \mathbb{R}^n$ be an operator that is second-order $L_2$-smooth with respect to the $\ell_2$ norm $\norm{\cdot}$. Let $\omega(\cdot,\cdot)$ denote the Bregman divergence of a function that is strongly convex with respect to the same norm $\norm{\cdot}$. Furthermore, suppose we are given $\Gamma, \Lambda, \Pi, \mu$ such that, for all iterates $x_i$, $x_{i + \frac{1}{2}}$ throughout the execution of Algorithm \ref{alg:MainAlgoApprox}, $\norm{\nabla U_{2,\xx_i}(\xx_{i+\frac{1}{2}})} \leq \Gamma$, $\norm{\xx_{i+\frac{1}{2}}-\xx_{i+1}} \leq \Lambda$, $\norm{F(\xx_i)} \leq \Pi$, and $\xx^\top \nabla F(\xx_i) \xx \geq \mu$ for all $\xx \in \calX$. In addition, let $\delta = \frac{ \epsilon}{2\Lambda\left(L_2 + \Gamma\right)}$. Then, Algorithm \ref{alg:MainAlgoApprox}, whereby \textsc{Approx-VI-Solve} is instantiated by \textsc{Approx-SO-VI-Solve} (Algorithm \ref{alg:SubproblemSolve}), returns $\hat{\xx}$ such that $\forall \xx \in \calX$,
\[
\langle\gg(\xx),\hat{\xx}-\xx\rangle \leq \epsilon ,
\]
in at most 
\[
\frac{16}{15} \left(2L_2\right)^{\nfrac{2}{3}}\frac{\omega(\xx,\xx_0)}{\epsilon^{\nfrac{2}{3}}}
\]
calls to \textsc{Approx-SO-VI-Solve} (Algorithm \ref{alg:SubproblemSolve}), each of which requires a single Schur decomposition and $O\left(\log\left(\frac{\left(L_2 + \Gamma\right)\Lambda\Pi}{\mu\epsilon}\right)\right)$ calls to a linear system solver in a quasi-upper-triangular system, for a total computational cost of $n^{\omega} + \tilde{O}(n)$, where $\omega \approx 2.3728$ is the matrix multiplication constant.
\end{theorem}

\begin{proof}
Invoking Theorem \ref{thm:approxSOerror} with our choice of $\delta = \frac{ \epsilon}{2\Lambda\left(L_2 + \Gamma\right)}$ implies that, for any iteration $i$, the output of Algorithm \ref{alg:SubproblemSolve} is such that 
\begin{equation*}
\langle U_{2,\xx_i}(\tilde{T}(\xx_i)), \tilde{T}(\xx_i)-\xx\rangle \leq \frac{\epsilon}{2}, \quad \forall \xx \in \rea^n.
\end{equation*}
The rest follows from Theorem \ref{thm:mainApproxThm}.

 Furthermore, the total number of calls to a linear system solver in a quasi-upper-triangular system is bounded $O\left(\log\left(\frac{\left(L_2 + \Gamma\right)\Lambda\Pi}{\mu\epsilon}\right)\right)$, which follows from Theorem \ref{thm:approxSOerror}, combined with our choice of $\delta$. 
\end{proof}

\end{document}